\pdfoutput=1
\RequirePackage{ifpdf}
\ifpdf % We are running pdfTeX in pdf mode
\documentclass[pdftex]{sigma}
\else
\documentclass{sigma}
\fi

\numberwithin{equation}{section}

\newtheorem{Theorem}{Theorem}[section]
\newtheorem{Corollary}[Theorem]{Corollary}
 { \theoremstyle{definition}
\newtheorem{Example}[Theorem]{Example}
\newtheorem{Remark}[Theorem]{Remark} }

\begin{document}
%\allowdisplaybreaks

\newcommand{\arXivNumber}{1901.09951}

\renewcommand{\thefootnote}{}

\renewcommand{\PaperNumber}{058}

\FirstPageHeading

\ShortArticleName{Linear Differential Systems with Small Coefficients: Various Types of Solvability}

\ArticleName{Linear Differential Systems with Small Coefficients:\\ Various Types of Solvability and their Verification\footnote{This paper is a~contribution to the Special Issue on Algebraic Methods in Dynamical Systems. The full collection is available at \href{https://www.emis.de/journals/SIGMA/AMDS2018.html}{https://www.emis.de/journals/SIGMA/AMDS2018.html}}}

\Author{Moulay A.~BARKATOU~$^\dag$ and Renat R.~GONTSOV~$^{\ddag\S}$}

\AuthorNameForHeading{M.A.~Barkatou and R.R.~Gontsov}

\Address{$^\dag$~Laboratoire XLIM (CNRS UMR 72 52), D\'epartement Math\'ematiques-Informatique,\\
\hphantom{$^\dag$}~Universit\'e de Limoges, Facult\'e des Sciences et Techniques, 123 avenue Albert Thomas,\\
\hphantom{$^\dag$}~F-87060 LIMOGES Cedex, France}
\EmailD{\href{mailto:email@address}{moulay.barkatou@unilim.fr}}

\Address{$^\ddag$~Institute for Information Transmission Problems RAS,\\
\hphantom{$^\ddag$}~Bolshoy Karetny per.~19, build.~1, Moscow 127051, Russia}
\Address{$^\S$~Moscow Power Engineering Institute, Krasnokazarmennaya 14, Moscow 111250, Russia}
\EmailD{\href{mailto:email@address}{gontsovrr@gmail.com}}

\ArticleDates{Received January 30, 2019, in final form July 31, 2019; Published online August 09, 2019}

\Abstract{We study the problem of solvability of linear differential systems with small coefficients in the Liouvillian sense (or, by generalized quadratures). For a general system, this problem is equivalent to that of solvability of the Lie algebra of the differential Galois group of the system. However, dependence of this Lie algebra on the system coefficients remains unknown. We show that for the particular class of systems with non-resonant irregular singular points that have {\it sufficiently small} coefficient matrices, the problem is reduced to that of solvability of the explicit Lie algebra generated by the coefficient matrices. This extends the corresponding Ilyashenko--Khovanskii theorem obtained for linear differential systems with Fuchsian singular points. We also give some examples illustrating the practical verification of the presented criteria of solvability by using general procedures implemented in Maple.}


\Classification{34M03; 34M25; 34M35; 34M50}

\renewcommand{\thefootnote}{\arabic{footnote}}
\setcounter{footnote}{0}

\section{Introduction}

Solvability of linear differential equations and systems in finite terms is a classical question of
differential Galois theory. It begins with the problem of solvability by generalized quadratures
(or, in the Liouvillian sense, which means the representability of all solutions of an equation in
terms of elementary or algebraic functions and their integrals, speaking informally) in the 1830's
in the works of Liouville on second order equations. Generalized in 1910 by Mordukhai--Boltovskii for
$n$th order equations, this was independently developed in a quite different way by Picard and Vessiot
who connected to an equation (system) a group, turned out to be a~linear algebraic group, called the
{\it differential Galois group}. They showed that solvability of the equation (system) in the Liouvillian
sense depends entirely on properties of this group or its Lie algebra. Namely, solvability holds if and
only if the identity component of the differential Galois group is solvable (which is equivalent to the
solvability of the Lie algebra of the group). Later Kolchin completed this theory by considering other,
more particular, types of solvability and their dependence on properties of the differential Galois group.

General methods for computing (in theory, at least) the differential Galois group have been proposed by several authors \cite{CoSi99,Fe15a,Hr02,Ho07b} in the last past decades. However, those methods have an extremely high computational complexity and are therefore not practical for nontrivial problems. Moreover it is not clear how the differential Galois group of a given system depends on the coefficients of the latter. More recently an algorithm for computing the Lie algebra of the differential Galois group has been proposed in~\cite{BCDW16}. This algorithm has a rather reasonable complexity, compared to the previous methods for computing the differential Galois group, but like the former it does not allow to clarify the relationship between the coefficients of the input system and its differential Lie algebra.

Though the main difficulty is that the differential Galois group (and its Lie algebra) of a~specific equation
depends on its coefficients very implicitly, in some cases it is possible to leave aside these implicit objects
(which are hard to compute) and obtain an answer to the question of solvability in terms of objects that are
determined by the coefficients of the equation explicitly. Here are some examples:
\begin{itemize}\itemsep=0pt
\item the list of hypergeometric equations solvable in the Liouvillian sense is completely known
(Scwharz--Kimura's list consisting of the fifteen families \cite{Ki}, see also \cite[Chapter~12, Section~1]{Zol2});
\item the Bessel equation is solvable in the Liouvillian sense if and only if its parameter is a~half-integer
(see \cite[Section~2.8]{MR} or \cite[Chapter~11, Section~1]{Zol2});
\item the equation $y''=(az^2+bz+c)y$, $a\ne0$, is solvable in the Liouvillian sense if and only if $(4ac-b^2)/4a^{3/2}$
is an odd integer (see~\cite{MMN});
\item more generally, an equation $y''=P(z)y$ with a polynomial $P$, is not solvable in the Liouvillian sense
if $\deg P$ is odd (see \cite{Ka}) whereas in the $(2n+1)$-dimensional space of such equations with $\deg P=2n$,
equations solvable in the Liouvillian sense form a union of countable number of algebraic varieties of dimension
$n+1$ each (see \cite{Vi,Zol1}).
\end{itemize}

A rather unexpected and maybe less known subclass of linear differential systems whose various types of solvability
can be checked in terms of explicit input data, is formed by systems with sufficiently small matrix coefficients.
The first result of this kind concerns Fuchsian systems. It was obtained by Ilyashenko, Khovanskii and published in
1974 in Russian, and is revised in~\cite{Kh1}, \cite[Chapter~6, Section~2.3]{Kh2}. This claims that for Fuchsian $(p\times p)$-systems
\begin{gather*}
\frac{{\rm d}y}{{\rm d}z}=\left(\sum_{i=1}^n\frac{B_i}{z-a_i}\right)y, \qquad y(z)\in{\mathbb C}^p, \qquad B_i\in{\rm Mat}(p,{\mathbb C}),
\end{gather*}
with fixed singular points $a_1,\ldots,a_n\in\mathbb C$ (and maybe $\infty$, if $\sum\limits_{i=1}^nB_i\ne0$) there exists an
$\varepsilon=\varepsilon(p,n)>0$ such that a criterion of solvability in the Liouvillian sense for a Fuchsian system
with $\|B_i\|<\varepsilon$ takes the following form: {\it the system is solvable if and only if all the matrices~$B_i$ can be
simultaneously reduced to a triangular form}. Using Kolchin's results, Ilyashenko and Khovanskii also obtained criteria for
other types of solvability of a Fuchsian system with small residue matrices, in terms of these matrices. Though, all these
criteria still remain not quite explicit in the sense that the $\varepsilon$ is expressed implicitely. This situation was
refined in~\cite{VG} where it was proved that for applying the above criterion of solvability it is sufficient for the eigenvalues of the matrices~$B_i$, called the {\it exponents} of the system, to be small enough rather than for the matrices themselves (this refinement had been conjectured earlier by Andrey Bolibrukh, see remark in~\cite[Section~6.2.3]{Kh1}). Moreover, an explicit bound for the exponents was given. Then similar criteria of solvability were obtained in \cite{GV} for a (non-resonant) irregular differential system with small {\it formal exponents}.

In the next section we recall these criteria for various types of solvability of a linear differential system with small (formal) exponents (Theorems~\ref{theorem1}, \ref{theorem1'}, \ref{theorem2} and \ref{theorem3}), previously giving necessary definitions. Then, in Section~\ref{section3}, we prove that the formal exponents at a non-resonant irregular singular point $z=a$ of a system
\begin{gather}\label{syst1}
\frac{{\rm d}y}{{\rm d}z}=B(z) y, \qquad B(z)=\frac1{(z-a)^{r+1}}\bigl(B^{(0)}+B^{(1)}(z-a)+\cdots+B^{(r)}(z-a)^r+\cdots\bigr)
\end{gather}
are small, if the coefficient matrices $B^{(1)},\ldots,B^{(r)}$ of the principal part of $B(z)$ at $a$ are small enough, while the leading term $B^{(0)}$ belongs to some compact subset of ${\rm Mat}(p,{\mathbb C})$. This allows us to formulate criteria of solvability inside an open set of non-resonant irregular differential systems (Theorem~\ref{theorem4}), looking at the problem theoretically, like Ilyashenko and Khovanskii did in the Fuchsian case.

From a practical point of view, as soon as we have a system whose (formal) exponents satisfy some numerical restrictions, which can be checked algorithmically, the question on solvability of the {\it implicit} Lie algebra of the differential Galois group of the system is reduced to the question on solvability of the {\it explicit} Lie algebra generated by the system coefficient matrices. This practical part of the problem is studied in Section~\ref{section4}, where the corresponding algorithms and examples are provided.

\section{Various types of solvability}\label{section2}

We consider a linear differential $(p\times p)$-system
\begin{gather}\label{syst2}
\frac{{\rm d}y}{{\rm d}z}=B(z) y,\qquad y(z)\in{\mathbb C}^p,
\end{gather}
defined on the whole Riemann sphere $\overline{\mathbb C}$, with the meromorphic coefficient matrix $B$ (whose entries are
thus rational functions).

One says that a solution $y$ of the system (\ref{syst2}) is {\it Liouvillian} if there is a tower of elementary
extensions
\begin{gather*}
{\mathbb C}(z)=F_0\subset F_1\subset\cdots\subset F_m
\end{gather*}
of the field ${\mathbb C}(z)$ of rational functions such that all components of $y$ belong to $F_m$. Here each
$F_{i+1}=F_i\langle x_i\rangle$ is a field extension of $F_i$ by an element $x_i$, which is either:
\begin{itemize}	\itemsep=0pt
	\item[--] an integral of some element in $F_i$, or
	\item[--] an exponential of integral of some element in $F_i$, or
	\item[--] algebraic over $F_i$.
\end{itemize}
The system is said to be solvable in the {\it Liouvillian sense} (or, by {\it generalized quadratures}),
if all its solutions are Liouvillian. There are other types of solvability studied by Kolchin \cite{Ko} from the point of
view of differential Galois theory, which are defined in analogy to solvability by generalized quadratures, and we leave
formal definitions to the reader. These are:
\begin{enumerate}	\itemsep=0pt
	\item[1)] solvability by integrals and algebraic functions;
	\item[2)] solvability by integrals;
	\item[3)] solvability by exponentials of integrals and algebraic functions;
	\item[4)] solvability by algebraic functions.
\end{enumerate}

We leave aside the differential Galois group of the system and the description of the above types of solvability in terms of
this group, since in the case of systems having small (formal) exponents we will deal with, this description is provided in
terms of the coefficient matrix. Let us recall some basic theory.

Assume that the system (\ref{syst2}) has singular points $a_1,\ldots,a_n\in\mathbb C$ of Poincar\'e rank $r_1,\ldots,r_n$
respectively and, for the simplicity of exposition, that the point $\infty$ is non-singular (though we will consider some example with a singular point at infinity; see Example~\ref{example1} below). This means that the coefficient matrix $B$ has the form
\begin{gather}\label{matrix2}
B(z)=\sum_{i=1}^n\left(\frac{B^{(0)}_i}{(z-a_i)^{r_i+1}}+\dots+\frac{B^{(r_i)}_i}{z-a_i}\right), \qquad
\sum_{i=1}^nB^{(r_i)}_i=0
\end{gather}
($B$ contains a polynomial part and/or $\sum\limits_{i=1}^nB^{(r_i)}_i\ne0$, if $\infty$ is a singular point).
\medskip

{\bf Genericity assumption.} We also restrict ourselves to the generic case, when each singular point $a_i$ is either
{\it Fuchsian} or {\it irregular non-resonant}. In the first case the Poincar\'e rank $r_i$ equals zero, while in the second
case it is positive and the eigenvalues of the leading term $B^{(0)}_i$ are {\it pairwise distinct}. Such a system will be
called {\it generic}.
\medskip

Near a Fuchsian singular point $z=a_i$, the system possesses a fundamental matrix $Y$ of the form
\begin{gather*}
Y(z)=U(z)(z-a_i)^A(z-a_i)^E,
\end{gather*}
where the matrix $U$ is holomorphically invertible at $a_i$ (that is, $\det U(a_i)\ne0$), $A$ is a diagonal integer matrix,
$E$ is a triangular matrix. The eigenvalues $\lambda_i^1,\ldots,\lambda_i^p$ of the triangular matrix $A+E$ are called the
{\it exponents} of the system at a Fuchsian singular point $a_i$ and they coincide with the eigenvalues of the residue matrix
$B^{(r_i)}_i=B^{(0)}_i$.

Near a non-resonant irregular singular point $z=a_i$, the system possesses a formal fundamental matrix $\widehat Y$ of the form
\begin{gather}\label{ffm}
\widehat Y(z)=\widehat F(z)(z-a_i)^{\Lambda} {\rm e}^{Q(1/(z-a_i))},
\end{gather}
where $\Lambda={\rm diag}\big(\lambda_i^1,\ldots,\lambda_i^p\big)$ is a diagonal matrix, $\widehat F$ is a matrix formal Taylor series
in $(z-a_i)$, with $\det\widehat F(a_i)\ne0$, and $Q$ is a polynomial. The numbers $\lambda_i^1,\ldots,\lambda_i^p$ are called
the {\it formal exponents} of the system at the non-resonant irregular singular point $a_i$, and the algorithm of their calculation will be recalled in the next section.

Now we give some theorems which follow from \cite{Go, GV} and which we will be based on further.

\begin{Theorem}\label{theorem1}Let the exponents of the generic system \eqref{syst2}, \eqref{matrix2} at all singular points satisfy the condition
\begin{gather}\label{cond1}
\operatorname{Re}\lambda_i^j>-1/n(p-1), \qquad i=1,\ldots,n, \qquad j=1,\ldots,p,
\end{gather}
furthermore, for every Fuchsian singular point $a_i$ each difference $\lambda_i^j-\lambda_i^k\not\in{\mathbb Q} \setminus{\mathbb Z}$. Then the system is solvable by generalized quadratures if and only if there exists a constant matrix $C\in{\rm GL}(p,\mathbb C)$ such that all the matrices $CB_i^{(l)}C^{-1}$ are upper triangular.
\end{Theorem}

\begin{Remark} By the genericity assumption, each {\it irregular} singular point of the system is non-resonant, whereas
{\it Fuchsian} singular points are not assumed to be non-resonant: integer differences $\lambda_i^j-\lambda_i^k$ of the
exponents at a Fuchsian singular point are not forbidden.
\end{Remark}

\begin{Theorem}\label{theorem1'}
The assertion of Theorem~{\rm \ref{theorem1}} remains valid if the condition \eqref{cond1} is replaced by
\begin{gather*}%\label{cond2}
\bigl|\operatorname{Re}\lambda_i^j-\operatorname{Re}\lambda_i^k\bigr|<1/n(p-1), \qquad i=1,\ldots,n, \qquad j,k=1,\ldots,p.
\end{gather*}
\end{Theorem}

According to the behaviour of solutions of a linear differential system near its irregular singular point, the system~(\ref{syst2}) with at least one irregular singular point is not solvable by integrals and algebraic functions.
For solvability by exponentials of integrals and algebraic functions, the following criterion holds.

\begin{Theorem}\label{theorem2}
Under the assumptions of Theorem~{\rm \ref{theorem1}}, the generic system~\eqref{syst2}, \eqref{matrix2} is solvable by exponentials of integrals and algebraic functions\footnote{Under the assumptions of the theorem, this type of solvability is equivalent to solvability by exponentials of integrals.} if and only if there exists a constant matrix $C\in{\rm GL}(p,\mathbb C)$ such that all the matrices $CB_i^{(l)}C^{-1}$ are diagonal.
\end{Theorem}

For a Fuchsian system, we additionally have the following criteria of solvability.

\begin{Theorem}\label{theorem3}
If all the singular points of the system \eqref{syst2}, \eqref{matrix2} are Fuchsian then, under the assumptions of Theorem~{\rm \ref{theorem1}}, this system is
\begin{enumerate}\itemsep=0pt
	\item[$1)$] solvable by integrals and algebraic functions\footnote{Under the assumptions of the theorem, this type of solvability is equivalent to solvability by integrals and radicals.} if and only if there exists
	$C\in{\rm GL}(p,\mathbb C)$ such that all the matrices $CB^{(0)}_iC^{-1}$ are triangular and the eigenvalues of each $B^{(0)}_i$ are rational numbers differing by an integer;
	\item[$2)$] solvable by integrals if and only if there exists $C\in{\rm GL}(p,\mathbb C)$ such that all the matrices $CB^{(0)}_iC^{-1}$ are triangular and their eigenvalues are equal to zero;
	\item[$3)$] solvable by algebraic functions\footnote{Under the assumptions of the theorem, this type of solvability is equivalent to solvability by radicals.} if and only if there exists $C\in{\rm GL}(p,\mathbb C)$ such that all the matrices $CB^{(0)}_iC^{-1}$ are diagonal and the eigenvalues of each $B^{(0)}_i$ are rational numbers differing by an integer.
\end{enumerate}
\end{Theorem}

\begin{Example}\label{example1}Let us give an example which illustrates that to expect the equivalence of solvability in the Liouvillian sense to the triangularizability of a system {\it via} a constant gauge transformation, one indeed need to put some restrictions on the (formal) system exponents.

Consider an equation
\begin{gather*}
\frac{{\rm d}^2u}{{\rm d}z^2}=\big(z^2+c\big)u,
\end{gather*}
where $c$ is a constant. Written in the form of a $(2\times2)$-system with respect to the vector of unknowns $y=(u,du/dz)^{\top}$, this becomes
\begin{gather*}
\frac{{\rm d}y}{{\rm d}z}=\left(\begin{matrix} 0 & 1 \\
 z^2+c & 0
 \end{matrix}\right)y,
\end{gather*}
or, in the variable $t=1/z$,
\begin{gather*}
\frac{{\rm d}y}{{\rm d}t}=\left(\begin{matrix} 0 & -1/t^2 \\
 -1/t^4-c/t^2 & 0
 \end{matrix}\right)y.
\end{gather*}
This is a system with the unique singular point $t=0$ of Poincar\'e rank~$3$ which is resonant though, since the leading term of the coefficient matrix is nilpotent. Thus we make the transformation $\tilde y=t^{{\rm diag}(0,1)}y$, under which the coefficient matrix~$B(t)$ of the system is changed as follows:
\begin{gather*}
A(t)=t^{{\rm diag}(0,1)}B(t) t^{-{\rm diag}(0,1)}+\frac1t {\rm diag}(0,1)=
\left(\begin{matrix} 0 & -1/t^3 \\
-1/t^3-c/t & 0
 \end{matrix}\right)+\left(\begin{matrix} 0 & 0 \\
 0 & 1/t
\end{matrix}\right) \\
\hphantom{A(t)}{} =\left(\begin{matrix} 0 & -1/t^3 \\
 -1/t^3 & 0
\end{matrix}\right)+\left(\begin{matrix} 0 & 0 \\
-c/t & 1/t
\end{matrix}\right)=\frac1{t^3}\big(A^{(0)}+A^{(2)}t^2\big).
\end{gather*}
The transformed system has two singular points: {\it non-resonant} irregular $t=0$ of Poincar\'e rank~$2$ (the eigenvalues of the leading term $A^{(0)}$ are $\pm1$) and Fuchsian $t=\infty$. The eigenvectors of the~$A^{(0)}$ are $e_1\mp e_2$, $e_{1,2}$ being the standard basic vectors $(1,0)^{\top}$, $(0,1)^{\top}$ of ${\mathbb C}^2$, whereas
\begin{gather*}
A^{(2)}(e_1-e_2)=(-c-1)e_2, \qquad A^{(2)}(e_1+e_2)=(-c+1)e_2.
\end{gather*}
Hence, the matrices $A^{(0)}$, $A^{(2)}$ are simultaneously triangularizable (have a common eigenvector) if and only if $c=1$ or $c=-1$. However, due to~\cite{MMN} (see the third example in our Introduction) the initial scalar equation and thus the transformed system, are solvable in the Liouvillian sense not only when $c=\pm1$ but also in the case of any other odd integer $c$. It remains to note that the exponents of the system at the Fuchsian point $t=\infty$ are $0$ and $-1$ (the eigenvalues of the matrix $-A^{(2)}$), that is, one of them does not satisfy the condition~(\ref{cond1}) of Theorem~\ref{theorem1} which in this case ($p=2$, $n=2$) looks like
\begin{gather*}
\operatorname{Re}\lambda_i^j>-1/2, \qquad i,j=1,2.
\end{gather*}
\end{Example}

\section{On the smallness of formal exponents}\label{section3}

Let us denote by $W\subset{\rm Mat}(p,{\mathbb C})$ a $($Zariski open$)$ subset of $(p\times p)$-matrices having pairwise
distinct eigenvalues. Assuming all the Poincar\'e ranks of the system (\ref{syst2}), (\ref{matrix2}) to be positive we regard it as a point in the set
\begin{gather*}
{\cal S}_{a_1,\ldots,a_n}^{r_1,\ldots,r_n}=\bigl\{ \big(B^{(0)}_1,\ldots,B^{(r_1)}_1\big),\ldots,\big(B^{(0)}_n,\ldots,B^{(r_n)}_n\big) \,|\, B^{(0)}_1,\ldots,B^{(0)}_n \in W \bigr\}.
\end{gather*}

Here we prove the following statement generalizing the Ilyashenko--Khovanskii criterion of solvability to the non-resonant
irregular case.

\begin{Theorem}\label{theorem4}
For any open disc $D\Subset W$ there exists $\delta=\delta(D,p,n)>0$ such that in a set
\begin{gather*}
{\cal S}_{a_1,\ldots,a_n}^{r_1,\ldots,r_n}(D,\delta)=\bigl\{ \big(B^{(0)}_i,\ldots,B^{(r_i)}_i\big)_{i=1}^n \in
{\cal S}_{a_1,\ldots,a_n}^{r_1,\ldots,r_n} \,|\, B^{(0)}_i \in D, \big\|B^{(1)}_i\big\|<\delta,\ldots,\big\|B^{(r_i)}_i\big\|<\delta\bigr\}
\end{gather*}
of systems with fixed singular points of fixed positive Poincar\'e ranks and sufficiently small non-leading coefficient matrices, systems solvable by generalized quadratures are those and only those whose matrices $B^{(0)}_i,\ldots,B^{(r_i)}_i$, $i=1,\ldots,n$, are simultaneously conjugate to triangular ones.
\end{Theorem}

This theorem will clearly follow from Theorem~\ref{theorem5} below and Theorem~\ref{theorem1}.

\begin{Theorem}\label{theorem5}
For any open disc $D\Subset W$ and any $\varepsilon>0$ there exists $\delta=\delta(D,\varepsilon)>0$ such that a local system~\eqref{syst1} with
\begin{gather*}
B^{(0)}\in D, \quad \big\|B^{(1)}\big\|<\delta, \quad \ldots, \quad \big\|B^{(r)}\big\|<\delta, \qquad r>0,
\end{gather*}
has the formal exponents $\lambda^1,\ldots,\lambda^p$ at the non-resonant irregular singular point $z=a$ satisfying the condition $|\lambda^j|<\varepsilon$, $j=1,\ldots,p$.
\end{Theorem}

\begin{proof} To prove the statement of the theorem, let us first recall the procedure of a formal transformation of the
system~(\ref{syst1}) to a diagonal form, the {\it splitting lemma} (from \cite[Section~11]{Wa}, and we assume that $a=0$ for simplicity).

Under a transformation $y=T(z)\tilde y$, the system is changed as follows
\begin{gather*}
\frac{{\rm d}\tilde y}{{\rm d}z}=A(z) \tilde y, \qquad A(z)=T^{-1}B(z)T-T^{-1}\frac{{\rm d}T}{{\rm d}z},
\end{gather*}
and a new coefficient matrix $A$ has the Laurent expansion
\begin{gather*}
A(z)=\frac1{z^{r+1}}\bigl(A^{(0)}+A^{(1)}z+\cdots\bigr).
\end{gather*}
One may always assume that $T(z)=I+T^{(1)}z+\cdots$ and $B^{(0)}=A^{(0)}={\rm diag}\big(\alpha^1,\ldots,\alpha^p\big)$. Then gathering the coefficients at each power $z^k$ from the relation
\begin{gather*}
T(z) z^{r+1}A(z)-z^{r+1}B(z) T(z)+z^{r+1}\frac{{\rm d}T}{{\rm d}z}=0,
\end{gather*}
we obtain
\begin{gather*}
T^{(k)}B^{(0)}-B^{(0)}T^{(k)}+A^{(k)}-B^{(k)}+\sum_{l=1}^{k-1}\big(T^{(l)}A^{(k-l)}-B^{(k-l)}T^{(l)}\big)+(k-r)T^{(k-r)}=0
\end{gather*}
(the last summand equals zero for $k\leqslant r$). There are two sets of unknowns in this system of matrix equations:
$A^{(k)}=\big(A^{(k)}_{ij}\big)$ and $T^{(k)}=\big(T^{(k)}_{ij}\big)$. Requiring all the $A^{(k)}$'s to be diagonal and assuming
$A^{(0)},\ldots,A^{(k-1)}$ and $T^{(0)},\ldots,T^{(k-1)}$ to be already found, one first obtains
\begin{gather*}
A^{(k)}_{jj}=B^{(k)}_{jj}-H^{(k)}_{jj},
\end{gather*}
where $H^{(k)}=\sum\limits_{l=1}^{k-1}\big(T^{(l)}A^{(k-l)}-B^{(k-l)}T^{(l)}\big)+(k-r)T^{(k-r)}$, and then
\begin{gather*}
T^{(k)}_{ij}=\frac1{\alpha^j-\alpha^i}\bigl(B^{(k)}_{ij}-H^{(k)}_{ij}\bigr),\qquad i\ne j, \qquad T^{(k)}_{jj}=0.
\end{gather*}

Since a diagonal system is integrated explicitly, one can see that the factors of the formal fundamental matrix
$\widehat Y(z)=\widehat F(z) z^{\Lambda}{\rm e}^{Q(1/z)}$ (see (\ref{ffm})) of the local system (\ref{syst1}) (we are
especially interested in $\Lambda$) are
\begin{gather*}
\widehat F(z)=T(z){\rm e}^{A^{(r+1)}z+A^{(r+2)}\frac{z^2}2+\cdots}, \qquad \Lambda=A^{(r)},\\
Q(1/z)=-\frac{A^{(0)}}{rz^r}-\frac{A^{(1)}}{(r-1)z^{r-1}}-\cdots-\frac{A^{(r-1)}}z.
\end{gather*}
Therefore, the formal exponents of the system (the eigenvalues of $\Lambda$) are
\begin{gather*}
\lambda^j=A^{(r)}_{jj}=B^{(r)}_{jj}-H^{(r)}_{jj}.
\end{gather*}
This implies the estimate
\begin{gather*}
\big|\lambda^j\big|\leqslant\big\|A^{(r)}\big\|\leqslant\big\|B^{(r)}\big\|+\big\|H^{(r)}\big\|
\end{gather*}
(we will use, for example, the matrix 1-norm $\|\cdot\|_1$ here).

Now we are ready to prove the smallness of the formal exponents $\lambda^j$'s
in the case of small coefficients $B^{(1)},\ldots,B^{(r)}$. For this, we should prove the smallness of $H^{(1)},\ldots,H^{(r)}$. Denoting $\rho=\min\limits_{i\ne j}\big|\alpha^j-\alpha^i\big|>0$ we will have
\begin{gather*}
\big|T^{(k)}_{ij}\big|\leqslant\frac1{\rho}\bigl(\big|B^{(k)}_{ij}\big|+\big|H^{(k)}_{ij}\big|\bigr)\quad\Longrightarrow
\quad \big\|T^{(k)}\big\|\leqslant\frac1{\rho}\bigl(\big\|B^{(k)}\big\|+\big\|H^{(k)}\big\|\bigr),
\end{gather*}
hence
\begin{gather*}
\big\|H^{(k)}\big\|\leqslant\sum_{l=1}^{k-1}\big\|T^{(l)}\big\|\bigl(\big\|A^{(k-l)}\big\|+\big\|B^{(k-l)}\big\|\bigr)\\
\hphantom{\big\|H^{(k)}\big\|}{} \leqslant
\sum_{l=1}^{k-1}\frac1{\rho}\bigl(\big\|B^{(l)}\big\|+\big\|H^{(l)}\big\|\bigr)\bigl(2\big\|B^{(k-l)}\big\|+\big\|H^{(k-l)}\big\|\bigr).
\end{gather*}
Assume that all $\big\|B^{(k)}\big\|<\delta$, $k=1,\ldots,r$. Then
\begin{gather}
\big\|H^{(k)}\big\|\leqslant\sum_{l=1}^{k-1}\frac1{\rho}\big(2\delta^2+\delta\bigl(2\big\|H^{(l)}\big\|+\big\|H^{(k-l)}\big\|\bigr)+
\big\|H^{(l)}\big\| \big\|H^{(k-l)}\big\|\big) \nonumber \\
\hphantom{\big\|H^{(k)}\big\|}{} =\frac1{\rho}\left(2(k-1)\delta^2+3\delta\sum_{l=1}^{k-1}\big\|H^{(l)}\big\|
+\sum_{l=1}^{k-1}\big\|H^{(l)}\big\|\big\|H^{(k-l)}\big\|\right). \label{normestimate}
\end{gather}
Thus having $H^{(1)}=0$, $\big\|H^{(2)}\big\|\leqslant2\delta^2/\rho$, one obtains
\begin{gather*}
\big\|H^{(r)}\big\|\leqslant\delta P_{r-1}(\delta/\rho) \qquad\mbox{and}\qquad
\big|\lambda^j\big|\leqslant\delta (1+P_{r-1}(\delta/\rho) ),
\end{gather*}
where $P_{r-1}$ is a polynomial of degree $r-1$. Indeed, since one already has
\begin{gather*}
\big\|H^{(2)}\big\|\leqslant\delta P_1(\delta/\rho), \qquad P_1(x)=2x,
\end{gather*}
making an inductive assumption
\begin{gather*}
\big\|H^{(k)}\big\|\leqslant\delta P_{k-1}(\delta/\rho) \qquad\mbox{for}\quad k=2,3,\ldots,r-1
\end{gather*}
(where $P_{k-1}$ is a polynomial of degree $k-1$) we deduce from~\eqref{normestimate} the required estimate for~$\big\|H^{(r)}\big\|$:
\begin{gather*}
\big\|H^{(r)}\big\|\leqslant\frac1{\rho}\left(2(r-1)\delta^2+3\delta\sum_{l=2}^{r-1}\delta P_{l-1}(\delta/\rho)+
\sum_{l=2}^{r-2}\delta^2 P_{l-1}(\delta/\rho)P_{r-l-1}(\delta/\rho)\right) \\
\hphantom{\big\|H^{(r)}\big\|}{}
= \delta\left(2(r-1)\frac{\delta}{\rho}+\sum_{l=2}^{r-1}3\frac{\delta}{\rho} P_{l-1}(\delta/\rho)+
\sum_{l=2}^{r-2}\frac{\delta}{\rho} P_{l-1}(\delta/\rho)P_{r-l-1}(\delta/\rho)\right)\\
\hphantom{\big\|H^{(r)}\big\|}{}= \delta P_{r-1}(\delta/\rho). \tag*{\qed}
\end{gather*}
\renewcommand{\qed}{}
\end{proof}

\begin{Remark}As follows from the above, the polynomials $P_k$ in the proof of Theorem~\ref{theorem5} are defined by the recurrence relations
\begin{gather*}
P_k(x)=2kx+\sum_{l=2}^k3x P_{l-1}(x)+\sum_{l=2}^{k-1}x P_{l-1}(x)P_{k-l}(x), \!\!\qquad k=2,\ldots,r-1, \!\qquad P_1(x)=2x,
\end{gather*}
in particular,
\begin{gather*}
	P_2(x)=4x+6x^2, \qquad P_3(x)=6x+18x^2+22x^3, \qquad\ldots, \\
	P_{r-1}(x)=2(r-1)x+3(r-1)(r-2)x^2+\dots+c_{r-1}x^{r-1}.
\end{gather*}
Therefore,
\begin{itemize}\itemsep=0pt
\item when $\varepsilon\rightarrow0$ and $\rho$ is fixed (or, more generally, $\varepsilon/\rho$ is bounded), then the leading term of the polynomial $\delta\bigl(1+P_{r-1}(\delta/\rho)\bigr)$ is $\delta$, and the condition
\begin{gather}\label{cond}
\delta\bigl(1+P_{r-1}(\delta/\rho)\bigr)<\varepsilon
\end{gather}
implies an asymptotic estimate $\delta(\varepsilon)=O(\varepsilon)$;
\item when $\varepsilon\rightarrow0$ and $\rho\rightarrow0$ such that $\varepsilon/\rho$ is unbounded, then the leading term of the polynomial $\delta\bigl(1+P_{r-1}(\delta/\rho)\bigr)$ satisfying~(\ref{cond}) is $c_{r-1}\delta^r/\rho^{r-1}$, and hence asympotically
\begin{gather*}
\delta(\varepsilon,\rho)=O\bigl(\varepsilon^{1/r}\rho^{(r-1)/r}\bigr).
\end{gather*}
\end{itemize}
\end{Remark}

There are two simple corollaries of Theorem~\ref{theorem4}.

\begin{Corollary}[for systems with singular points of Poincar\'e rank 1]
In a set
\begin{gather*}
{\cal S}_{a_1,\ldots,a_n}^{\;\;1,\ldots,1}\bigl(1/n(p-1)\bigr)=\bigl\{ \big(B^{(0)}_i,B^{(1)}_i\big)_{i=1}^n \in
{\cal S}_{a_1,\ldots,a_n}^{\;\;1,\ldots,1} \,|\, \big\|B^{(1)}_i\big\|< 1/n(p-1) \bigr\},
\end{gather*}
systems solvable by generalized quadratures are those and only those whose matrices $B^{(0)}_i$, $B^{(1)}_i$, $i=1,\ldots,n$, are simultaneously conjugate to triangular ones.
\end{Corollary}

\begin{proof} If all the Poincar\'e ranks $r_i=1$, then $\big|\lambda_i^j\big|\leqslant\big\|B^{(1)}_i\big\|$ and thus one has no necessity to restrict the leading terms $B^{(0)}_i$ on some disc $D\Subset W$ and may take any $B^{(0)}_i\in W$.
\end{proof}

\begin{Corollary}\label{corollary2} System \eqref{syst2} with the coefficient matrix $B$ of the form
\begin{gather*}
B(z)=\sum_{i=1}^n\frac{B^{(0)}_i}{(z-a_i)^{r_i+1}}, \qquad B^{(0)}_i\in W, \qquad r_i>0,
\end{gather*}
is solvable by generalized quadratures if and only if all the matrices $B^{(0)}_1,\ldots,B^{(0)}_n$ are simultaneously conjugate to triangular ones.
\end{Corollary}

\begin{proof} In the case $B^{(1)}_i=\dots=B^{(r_i)}_i=0$, $i=1,\ldots,n$, all the formal exponents equal zero, and there is again no need to restrict the matrices~$B^{(0)}_i$ on some disc $D\Subset W$.
\end{proof}

\section{Practical remarks and examples}\label{section4}

We conclude by noting that the results presented in this paper can be turned into an algorithm which allows to check the solvability of a given concrete system. The main steps can be summarized as follows.
\begin{enumerate}	\itemsep=0pt
	\item Compute the formal exponents at each singular point. 	The practical calculation of the formal exponents presented in the previous section, is already implemented in Maple as a~part of the general program ISOLDE \cite{BP} based on algorithms from~\cite{Ba}.
	\item Check the required smallness of the formal exponents of the system under consideration.
	\item The simultaneous triangularizability of a set of matrices is equivalent to the solvability of the Lie algebra generated by them. This can be checked using the package {\it LieAlgebras} of Maple.\footnote{Note that checking the simultaneous triangularizability of a given set of matrices can be done using more direct algorithms (that is, without recourse to Lie algebra computations). Those algorithms consist in trying first to find a common eigenvector of the input matrices and then to proceed recursively (this is a work in progress~\cite{BC}).}
\end{enumerate}

In order to illustrate the various steps of our algorithm we give two examples of computations performed by using our implementation in Maple.

\begin{Example} We consider a $(3\times3)$-system (\ref{syst2}) with the coefficient matrix
\begin{gather*}
B(z)= M_1/z^3+M_4/z+M_2/(z-1)^2+M_5/(z-1)+M_3/(z+1)^2+M_6/(z+1),
\end{gather*}
where the matrices $M_i$ are given by
\begin{gather*}
M_1=\left( \begin{matrix} -5&-4&-4\\ 17&14&13
\\ -10&-8&-7\end{matrix} \right), \qquad M_2= \left( \begin{matrix} -6&-5&-5\\ 23&17&15
\\ -14&-9&-7\end{matrix} \right),
\qquad M_3= \left(\begin{matrix}1&1&1\\ -11&-7&-6
\\ 8&4&3\end{matrix} \right),
\\
M_4=\left(\begin{matrix} 2 a-c&a-c&a-c\\ -3-6 a
+5 c&-2-3 a+5 c&-2-3 a+5 c\\ 3+4 a-3 c&2+2 a
-3 c&2+2 a-3 c\end{matrix}\right),
\\
M_5=\left(\begin{matrix} 0&0&0\\ b+2&-b+1&-2 b+1
\\ -b-2&b-1&2 b-1\end{matrix}\right),
\\
M_6=\left(\begin{matrix} -2 a+c&-a+c&-a+c\\ -b+6
 a+1-5 c&b+3 a+1-5 c&2 b+3 a+1-5 c\\ b-4 a-1
+3 c&-b-2 a-1+3 c&-2 b-2 a-1+3 c\end{matrix} \right)
\end{gather*}
and $a$, $b$ and $c$ are parameters. It has three singular points $a_1=0$, $a_2=1$, $a_3=-1$ of Poincar\'e rank $r_1=2$, $r_2=1$, $r_3=1$, respectively. The point $z=\infty$ is non-singular, since $M_4+M_5+M_6=0$.

Our implementation allows to check that the three singularities are non-resonant and gives the formal exponents for each of them.
\begin{verbatim} > read "/Users/barkatou/Desktop/Spliting/split":
\end{verbatim}

We apply the splitting lemma at $z=a_1$ up to $k=2$.
\begin{verbatim}
> splitlemma(B, z=0, 2, T);
\end{verbatim}
We get the equivalent matrix
\begin{gather*}
\left[ \begin{matrix} -\dfrac1{z^3}+\dfrac az&0&0\\ 0&\dfrac1{z^3}&0\\ 0&0&\dfrac2{z^3}+
\dfrac cz\end{matrix} \right]+\mbox{regular part}.
\end{gather*}
The formal exponents at $z=a_1$ are then $\lambda_1^1=0$, $\lambda_1^2=a$, $\lambda_1^3=c$.
\medskip

We apply now the splitting lemma at $z=a_2$ up to $k=1$.
\begin{verbatim}
> splitlemma(B, z=1, 1, T);
\end{verbatim}
We get the equivalent matrix
\begin{gather*}
\left[ \begin{matrix} -\dfrac1{(z-1)^2}&0&0
\\ 0&\dfrac2{(z-1)^2}+\dfrac b{z-1}&0
\\ 0&0&\dfrac3{(z-1)^2}\end{matrix}
\right]+\mbox{regular part}.
\end{gather*}
The formal exponents at $z=a_2$ are then $\lambda_2^1=0$, $\lambda_2^2=0$, $\lambda_2^3=b$.

Finally, we apply the splitting lemma at $z=a_3$ up to $k=1$.
\begin{verbatim}
> splitlemma(B, z=-1, 1, T);
\end{verbatim}
We get the equivalent matrix
\begin{gather*}
\left[ \begin{matrix} - \dfrac {a}{z+1} &0&0\\ 0
&-\dfrac2{(z+1)^2}-\dfrac c{z+1}&0\\ 0
&0&-\dfrac1{(z+1)^2}-\dfrac b{z+1}\end{matrix} \right]+\mbox{regular part}.
\end{gather*}
The formal exponents at $z=a_3$ are then $\lambda_3^1=-a$, $\lambda_3^2=-c$, $\lambda_3^3=-b$.

Note that in our case ($p=3$, $n=3$) the condition (\ref{cond1}) of Theorem~\ref{theorem1} is as follows:
\begin{gather*}
\operatorname{Re}\lambda_i^j>-1/6, \qquad i,j=1,2,3.
\end{gather*}
This is equivalent here to the condition
\begin{gather*}
|\operatorname{Re}\lambda|<1/6, \qquad \mbox{for} \quad \lambda \in \{a,b,c\}.
\end{gather*}

Now let us check in which case the matrices $M_i$, $i=1,\dots,6$, are simultaneously triangularizable. For this, we can first use the {\em LieAlgebras} package of Maple.
\begin{enumerate}	\itemsep=0pt
	\item First one computes the Lie algebra $L$ generated by the set $S=\{M_1,\ldots,M_6\}$.
	\item Then one can use the {\em Query} function with the argument ``{\em Solvable}'' to check the solvability of~$L$.
\end{enumerate}
\begin{verbatim}
> with(DifferentialGeometry):
> with(LieAlgebras):
> S:=[seq(M_i, i = 1 .. 6)]:
> L := MatrixLieAlgebra(S);
\end{verbatim}
\begin{gather*}\left[ \left[ \begin{matrix} -5&-4&-4 \\ 17&14&13
	\\ -10&-8&-7 \end{matrix} \right] ,
	\left[\begin{matrix} -6&-5&-5 \\ 23&17&15
	\\ -14&-9&-7 \end{matrix} \right] ,
	\left[\begin{matrix} 1&1&1 \\ -11&-7&-6
	\\ 8&4&3 \end{matrix} \right],
	\left[ \begin{matrix} -1&-1&-1
	\\ 14&10&10 \\ -13&-9&-9
	\end{matrix} \right], \right.\\
	\left.\left[ \begin{matrix} 2 a-c&a-c&a-c \\ -3-6 a+5 c&-2-3 a+5 c&-2-3 a+5 c
	\\ 3+4 a-3 c&2+2 a-3 c&2+2 a-3 c
	\end{matrix} \right] , \left[ \begin{matrix} 0&0&0
	\\ b+2&-b+1&-2 b+1\\ -b-2&b-1&2 b-1
	\end{matrix}\right] \right].
 \end{gather*}
\begin{verbatim}
> LD1:=LieAlgebraData(L,Alg1):
> DGsetup(LD1):
We check the solvability of L
> Query("Solvable");
true
\end{verbatim}

This tells us that the Lie algebra $L$ is sovable (apparently for all $a$, $b$, $c$) but it does not give a transformation~$P$ that simultaneously triangularizes the matrices $M_i$. Such a transformation can be directly obtained using the implementation from~\cite{BC}.
\begin{verbatim}
> read "/Users/barkatou/Desktop/Simul-Triang/utilitaires_reduction.mpl":
> read "/Users/barkatou/Desktop/Simul-Triang/SimultTriang.mpl":
> P := SimultaneousTriangularization([seq(M[i], i = 1 .. 6)]);
\end{verbatim}
\begin{gather*}
P:=\left[ \begin{matrix} 0&-1&1\\ 1&1&0
\\ -1&0&0\end{matrix} \right].
\end{gather*}
One can check that indeed the matrices $P^{-1}M_iP$ are triangular:
\begin{verbatim}
> seq(1/P . M[i] . P, i = 1 .. 6);
\end{verbatim}
 \begin{gather*}\left[ \begin{matrix} 1&-2&10\\ 0&-1&7
	\\ 0&0&2\end{matrix} \right],
	\left[ \begin{matrix} 2&-5&14\\ 0&-1&9
	\\ 0&0&3\end{matrix} \right],
	\left[ \begin{matrix} -1&4&-8\\ 0&0&-3
	\\ 0&0&-2\end{matrix} \right], \\
	\left[ \begin{matrix} 0&1+2 a&-3-4 a+3 c\\ 0
	&a&-2 a+2 c\\ 0&0&c\end{matrix} \right]
	, \left[ \begin{matrix} b&-2 b-1&b+2\\ 0&0&0
	\\ 0&0&0\end{matrix} \right],
	\left[ \begin{matrix} -b&2 b-2 a&-b+4 a+1-3 c
	\\ 0&-a&2 a-2 c\\ 0&0&-c
	\end{matrix} \right].
	\end{gather*}

In conclusion, the system in this example is solvable in the Liouvillian sense without any restriction on the parameters $a$, $b$, $c$. On the other hand, the matrices $M_i$, $i=1,\dots,6$, are never simultaneously diagonalizable since, for example, $M_1$, $M_2$ do not commute:
\begin{gather*} [M_1,M_2]=\left( \begin{matrix} -1&-1&-1\\ 14&10&10
\\ -13&-9&-9\end{matrix} \right). \end{gather*}
Hence, due to Theorem~\ref{theorem2} we can assert that the system is not solvable by exponentials of integrals and algebraic functions whenever $|\operatorname{Re}\lambda|<1/6$, for $\lambda\in\{a,b,c\}$.
\end{Example}

\begin{Example} We consider a $(3\times3)$-system (\ref{syst2}) with the coefficient matrix
\begin{gather*}
B(z)= M_1/(z-a_1)^{1+r_1}+M_2/(z-a_2)^{1+r_2}+M_3/(z-a_3)^{1+r_3},
\end{gather*}
where the matrices $M_i$ are given by
\begin{gather*}
M_1=\left( \begin{matrix} 1&0&0\\ 0&-1&0
\\ 0&0&2\end{matrix} \right),
\qquad M_2 = \left( \begin{matrix} 0&0&0\\ 3 a&3+b&1
\\ -3 ab&-{b}^{2}-5 b&-2-b\end{matrix} \right),
\qquad M_3 = \left( \begin{matrix} -1&0&0\\ 0&4&0
\\ -2&0&1\end{matrix} \right),
\end{gather*}
$r_i$ are positive integers and $a$, $b$ are parameters. The matrix~$M_2$ has three distinct eigenvalues: $0$, $3$, $-2$.
Hence the system has three non-resonant singularities and is of the form assumed in Corollary~\ref{corollary2}.
\begin{verbatim}
> with(DifferentialGeometry):
> with(LieAlgebras):
> S:=[seq(M_i, i = 1 .. 3)]:
> L := MatrixLieAlgebra(S);
\end{verbatim}
\begin{gather*} \left[ \left[ \begin{matrix} 1&0&0\\ 0&-1&0
		\\ 0&0&2\end{matrix} \right] , \left[ \begin{matrix} 0&0&0\\ 3 a&3+b&1\\ -3 ab&
		-{b}^{2}-5 b&-2-b\end{matrix} \right] , \left[ \begin{matrix} -1
		&0&0\\ 0&4&0\\ -2&0&1\end{matrix}
		\right] , \left[ \begin{matrix} 0&0&0\\ -6 a&0
		&-3\\ -3 ab&-3 {b}^{2}-15 b&0\end{matrix}
		\right] ,\right.\\
		\left.\left[ \begin{matrix} 0&0&0\\ 0&0&0
		\\ -2&0&0\end{matrix} \right] , \left[
		\begin{matrix} 0&0&0\\ -15 a-2&0&-3
		\\ 6 ab+2 b+4&-3 {b}^{2}-15 b&0\end{matrix}
		\right] , \left[ \begin{matrix} 0&0&0\\ 12 a&0
		&9\\ -3 ab&-9 {b}^{2}-45 b&0\end{matrix} \right] \right].
\end{gather*}
\begin{verbatim}
> LD1:=LieAlgebraData(L,Alg1):
> DGsetup(LD1):
We check the solvability of L
> Query("Solvable");
false
\end{verbatim}
This should be interpreted as that the Lie algebra $L$ is not solvable (and hence our three matrices~$M_i$ are not simultaneously triangularizable) for {\em generic} values of $a$, $b$. Therefore, the system is not solvable by generalized quadratures for generic values of~$a$, $b$.

Our implementation \cite{BC} gives the same answer:
\begin{verbatim}
> read "/Users/barkatou/Desktop/Simul-Triang/utilitaires_reduction.mpl":
> read "/Users/barkatou/Desktop/Simul-Triang/SimultTriang.mpl":
> P := SimultaneousTriangularization([seq(M[i], i = 1 .. 3)]);

"The matrices are not simultaneously triangularizable!"
\end{verbatim}

Now one can check directly that for
$b=-5$ or $b=0$, the matrices $M_i$ are simultaneously triangularizable.
Indeed, if we consider the following matrix
\begin{gather*}
P:= \left[ \begin{matrix} 0&0&1\\ -9&0&0
\\ 0&1&0\end{matrix} \right]
\end{gather*}
and compute the matrices $P^{-1}M_iP$, we find that
\begin{verbatim}
> seq(1/P . M[i] . P, i = 1 .. 3);
\end{verbatim}
\begin{gather*}
\left[ \begin{matrix} -1&0&0\\ 0&2&0
\\ 0&0&1\end{matrix} \right],
\left[ \begin{matrix} 3+b&-1/9&-a/3\\ 9 b
 ( b+5 ) &-2-b&-3 ab\\ 0&0&0\end{matrix}
\right],
\left[ \begin{matrix} 4&0&0\\ 0&1&-2
\\ 0&0&-1\end{matrix} \right],
\end{gather*}
from which we see that the matrices~$M_i$ are simultaneously triangularizable when $b(b+5)=0$. This means that for $b=-5$ or $b=0$, the system is solvable by generalized quadratures and, on the other hand, is not solvable by exponentials of integrals and algebraic functions (since the matrices $M_i$ are not simultaneously diagonalizable: $[M_1,M_2]\ne0$).
\end{Example}

Note that for this kind of examples where the matrices~$M_i$ depend on some parameters, one question naturally arises, namely: {\it is it possible to find the equations for the solvability locus in the parameter space?}

The answer to this question is positive. Indeed, one can use {\em Cartan's criterion for solvabi\-li\-ty}~\cite{Hu72} which states that a matrix Lie algebra $\mathfrak{g}$ is solvable (in characteristic zero) if and only if $K(\mathfrak{g}, [\mathfrak{g},\mathfrak{g}]) =0$, where $K$ designates the Killing form\footnote{$K(u,v)=\operatorname{Tr}({\rm ad}(u){\rm ad}(v))$.} of~$\mathfrak{g}$. This is also equivalent to $\operatorname{Tr}(uv)=0$ for all $u \in \mathfrak{g}$ and all $v \in [\mathfrak{g},\mathfrak{g}]$, the derived Lie algebra of $\mathfrak{g}$.

In order to check in practice that a given Lie algebra $\mathfrak{g}$ is solvable, one could proceed as follows.
Let $\{u_1, \dots, u_r\}$, resp. $\{v_1, \dots, v_s\}$, be a basis of $\mathfrak{g}$, resp. of $[\mathfrak{g},\mathfrak{g}]$. Using the package {\it LieAlgebras} of Maple one can compute $f_{ij}:=K(u_i,v_j)$ for $i =1,\dots, r$, $j=1,\dots, s$, and check whether all the $f_{ij}$'s are zero or not. In the case when the elements of $\mathfrak{g}$ are matrices, say, with coefficients in ${\mathbb C}(p_1, \dots, p_t)$ for some parameters $p_1, \dots, p_t$, the $f_{ij}$'s are then rational functions of $p_1, \dots, p_t$ and the solvability locus in the parameter space is given by the zero set of the numerators of the $f_{ij}$'s. In our example, the Maple computation gives the solvability locus determined by $b(b+5)=0$.

{\bf Concluding remark.} As we have already mentioned in Introduction, in practice, to check the solvability by generalized quadratures of a given linear differential system ${\rm d}y/{\rm d}z =B(z)y$, it is sufficient (but might be very costly) to compute the Lie algebra of the corresponding differential Galois group (for example, by the algorithm in~\cite{BCDW16}) and this works, in theory at least, for a general system. Now for the class of systems with small coefficients it is much better (from the computational point of view) to apply the method of the present paper, namely computing the Lie algebra defined by the coefficients of the matrix~$B(z)$. Indeed, let us denote by~$\mathfrak{g} \subset {\mathfrak{gl}}_n({\mathbb C})$ the Lie algebra of the differential Galois group of our system and by~$\mathfrak{h}$ the Lie algebra generated by the coefficient matrix. The latter is defined as follows: if $B(z) = f_1M_1 + \dots+ f_sM_s$, where the~$M_i$'s are constant matrices and the $f_i$'s are rational functions in ${\mathbb C}(z)$ that form a basis of the $\mathbb C$-vector space generated by the entries of~$B(z)$, then $\mathfrak{h}$ is the matrix Lie algebra generated by the~$M_i$'s. One has that $\mathfrak{g}$ is always a Lie subalgebra of $\mathfrak{h}$. In general this inclusion is strict, and when the equality holds the differential system is called in {\em reduced form} (see~\cite{BCDW16} and references therein). Now, due to the Kolchin--Kovaci\'c theorem, there exists a gauge transformation $y = P \tilde{y}$, with $P$ having its entries in the algebraic closure of~${\mathbb C}(z)$, that takes the given system ${\rm d}y/{\rm d}z = B(z) y$ into an equivalent one ${\rm d}\tilde{y}/{\rm d}z=\tilde{B}(z) \tilde{y}$ that is in reduced form. The algorithm in~\cite{BCDW16} computes such a transformation and hence gets $\mathfrak{g}$ as the Lie algebra defined by $\tilde{B}(z)$. However computing a reduced form is not an easy task at all, since this is equivalent to finding the differential Lie algebra~$\mathfrak{g}$. Now, for systems with small coefficients what we propose is considering the set of the matrices $B_i^{(j)}$ involved in the partial fraction decomposition of~$B(z)$ and testing whether they are
simultaneously triangularizable over~$\mathbb C$ or not. For this (even if it is not necessary) one can form the Lie algebra $\mathfrak{h}'$ generated by those matrices and test its solvability. It is not difficult to see that in fact $\mathfrak{h}' = \mathfrak{h}$ (as defined above). Hence for systems with small coefficients there is no need to compute the differential Lie algebra~$\mathfrak{g}$, in other words there is no need to compute a reduced form of the system, one has only to check the solvability of $\mathfrak{h}'$ which is obtained directly from the coefficients of the system.

\subsection*{Acknowledgements}

The authors are grateful to Thomas Cluzeau for helpful discussions about the problem of simultaneous triangularizability of a set of matrices. They also thank the referee for his/her nice suggestions which have refined the text. The work of R.G.\ was partially supported by the Russian Foundation for Basic Research (projects 16-51-1500005 and 17-01-00515).

\pdfbookmark[1]{References}{ref}
\LastPageEnding

\end{document}